\documentclass[a4paper, 11pt]{article}
\usepackage[top=30truemm, bottom=30truemm, left=20truemm, right=20truemm]{geometry}
\usepackage{amsthm}
\usepackage{amsmath}
\usepackage{amssymb}
\usepackage{amsfonts}
\usepackage{mathtools}
\usepackage{fancyhdr}
\usepackage{xcolor}
\usepackage{bigints}
\usepackage{enumerate}
\usepackage{indentfirst}
\usepackage{mathrsfs}
\usepackage{mathtools}
\usepackage{empheq}
\usepackage{physics}
\usepackage{bm}
\usepackage{enumitem}
\usepackage[hidelinks]{hyperref}
\newtheorem{theorem}{Theorem}[section]

\newtheorem{lemma}[theorem]{Lemma}
\newtheorem{proposition}[theorem]{Proposition}

\theoremstyle{definition}
\newtheorem{definition}[theorem]{Definition}
\newtheorem{remark}[theorem]{Remark}

\setlist[itemize]{itemsep=0pt, parsep=0.5\baselineskip, topsep=0.5\baselineskip}
\allowdisplaybreaks

\title{
\vspace*{20mm} \huge
An initial-boundary value problem describing moisture transport in porous media: existence of strong solutions and an error estimate for a finite volume scheme\\
\vspace{15mm}
}

\author{
{\Large Akiko Morimura}\footnote{Corresponding author}\\
\vspace{-1.5mm}\\
Division of Mathematical and Physical Science, Graduate School of Science,\\
Japan Women's University\\
2-8-1 Mejirodai, Bunkyoku, Tokyo, 112-8681, Japan\\
(email: \texttt{m1716096ma@ug.jwu.ac.jp})\\
\vspace*{2mm} \\
and\\
\vspace*{2mm} \\
{\Large Toyohiko Aiki}\\
\vspace{-1.5mm}\\
Department of Mathematics, Physics and Computer Science, Faculty of Science,\\
Japan Women's University\\
2-8-1 Mejirodai, Bunkyoku, Tokyo, 112-8681, Japan\\
(email: \texttt{aikit@fc.jwu.ac.jp})
}

\date{}

\begin{document}
\maketitle
\renewcommand{\thefootnote}{\fnsymbol{footnote}}

\begin{abstract}
We consider an initial-boundary value problem motivated by a mathematical model of moisture transport in porous media.
We establish the existence of strong solutions and provide an error estimate for the approximate solutions constructed by the finite volume method.
In the proof of the error estimate, the Gagliardo--Nirenberg type inequality for the difference between a continuous function and a piecewise constant function plays an important role.
\end{abstract}


\newpage
\section{Introduction}
In this paper, we consider the following initial-boundary value problem (P)($p, v_0$):
\begin{alignat}{2}
\label{P_eq}
&\partial_t h(v) = \partial_x \left( \partial_x v + b(v) p \right)
&&\qquad \text{in $Q(T)$}, \\
\label{P_BC}
&\partial_x v + b(v) p  = 0
&&\qquad \text{at $x = 0, 1$ and for $t \in (0, T)$},\\
\label{P_IC}
& v(0, x) = v_0(x)
&&\qquad \text{for $x \in (0, 1)$},
\end{alignat}
where $Q(T) \coloneqq (0,T) \times (0,1)$; $v \colon [0, T) \times (0, 1) \to \mathbb{R}$ is an unknown function; $v_0: (0, 1) \to \mathbb{R}$ is a given datum at $t=0$; $h, b \colon \mathbb{R} \to \mathbb{R}$ and $p \colon (0, T) \times [0, 1] \to \mathbb{R}$ are given functions.
We impose the following assumptions (A1) and (A2) on $h$ and $b$, respectively:
\begin{itemize}
\item[(A1)] $h\in C^2(\mathbb{R})$ with
$C_{h, 1} \leq h' \leq C_{h, 2}$ and 
$|h''| \leq C_{h, 2}$ on $\mathbb{R}$;
\item[(A2)] $b \in C^2(\mathbb{R})$ with
$C_{b, 1} \leq b \leq C_{b, 2}$
and $|b'| , |b''| \leq C_{b, 2}$ on $\mathbb{R}$
\end{itemize}
for some positive constants $C_{h, 1}$, $C_{h, 2}$, $C_{b, 1}$, and $C_{b, 2}$.
The problem (P)($p, v_0$) is based on a mathematical model of moisture transport in porous media proposed by Green--Dabiri--Weinaug--Prill \cite{G-D-W-P} and Fukui--Iba--Hokoi--Ogura \cite{F-I-H-O}.
The original model is a system of partial differential equations describing the mass conservation laws of water and air.
In our previous work \cite{M-A-2}, due to mathematical difficulties arising from the system, we focused on one of the mass conservation laws and introduced the following initial-boundary value problem:
\begin{alignat}{2}
\label{PP_eq}
&\partial_t \psi(u) = \partial_x \left( \lambda(u) \partial_x ( u +  P)  \right) 
&&\qquad \text{in $Q(T)$},\\
\label{PP_BC}
&\lambda(u) \partial_x (u + P) = 0 
&&\qquad \text{at $x = 0, 1$ and for $t \in (0, T)$},\\
\label{PP_IC}
& u(0, x) = u_0(x)
&&\qquad \text{for $x \in (0, 1)$}.
\end{alignat}
Here, $u \colon [0, T) \times (0, 1) \to \mathbb{R}$ is an unknown function; $u_0 \colon (0, 1) \to \mathbb{R}$ is a given datum at $t=0$; $\psi, \lambda \colon \mathbb{R} \to \mathbb{R}$ and $P \colon (0, T) \times [0, 1] \to \mathbb{R}$ are given.
%We note that $p = \partial_x P$.
Under a certain setting, the problem \eqref{PP_eq}--\eqref{PP_IC} can be transformed into (P)($p, v_0$).
For details, see \cite{M-A-1, M-A-2}.
We remark that $\psi'$ and $\lambda$ are not degenerate due to assumptions similar to (A1) and (A2) imposed on $\psi$ and $\lambda$ in \cite{M-A-1}.
Hence, our target equation does not cover some models for unsaturated flow in porous media, such as the van Genuchten model.

We next review some known results for (P)($p, v_0$) and related problems.
Roub\'i\v{c}ek \cite{Roubicek} established the existence of weak solutions to (P)($p, v_0$) by using the theory of pseudo-monotone operators.
We note that several technical conditions must be verified to employ this approach.
In our previous work \cite{M-A-1}, we proved the uniqueness of weak solutions to (P)($p, v_0$) and the convergence of approximate solutions constructed by the finite volume method (FVM) to the weak solution.
We also gave an alternative proof of the existence of weak solutions to (P)($p, v_0$) in the framework of the FVM.
An equation of the form \eqref{PP_eq} is known as the Richards-type equation, which has been studied in a large literature.
Mitra--Vohral\'{i}k \cite{M-V}, Dolej\v{s}\'{i}--Shin--Vlas\'{a}k \cite{D-S-V}, and Stokke--Mitra--Storvik--Both--Radu \cite{S-M-S-B-R} proposed numerical schemes for the Richards-type equation and provided estimates for numerical solutions in terms of certain indicators.
However, these estimates do not imply the convergence of the numerical solutions to a solution of the Richards-type equation.
Merz--Rybka \cite{M-R} and Berardi--Difonzo \cite{B-D} established the existence and uniqueness of strong solutions to the Richards-type equation in the case where $P$ is independent of the time variable.
See also \cite{M-A-1} for remarks on the dependence of $P$ on the space and time variables in previous works.
We note that, to the best of our knowledge, there is no result on the existence of strong solutions to the Richards-type equation when $P$ depends on both the space and time variables.
For a model of water transport, Eymard--Gallou\"{e}t--Herbin \cite{E-G-H} proved the existence of weak solutions and the convergence of approximate solutions to a weak solution by applying the FVM.

In this paper, as a continuation of \cite{M-A-1}, we establish the existence of strong solutions to (P)($p, v_0$) in the framework of the FVM.
Our result covers the Richards-type equation with $P$ depending on both the space and time variables.
We note that the uniqueness of strong solutions to (P)($p, v_0$) follows from that of weak solutions proved in \cite{M-A-1}.
Moreover, thanks to the regularity of the strong solution, we provide not only the convergence of approximate solutions constructed by the FVM to the strong solution but also an error estimate for this convergence.
Our main results further strengthen the effectiveness of the FVM for (P)($p, v_0$).

The rest of this paper is organized as follows.
In the next section, we state our main results: Theorems \ref{thm_ESS} and \ref{thm_EE}.
In Sections \ref{sec_ESS} and \ref{sec_EE}, we give the proofs of Theorems \ref{thm_ESS} and \ref{thm_EE}, respectively.

\section{Notation and main results}
In this paper, we use
\begin{align*}
H \coloneqq L^2(0,1), \qquad
X \coloneqq H^1(0,1), \qquad
|\cdot|_{H} \coloneqq |\cdot|_{L^{2}(0, 1)}, \qquad
|\cdot|_{X} \coloneqq |\cdot|_{H^{1}(0, 1)}.
\end{align*}
This paper aims to establish the existence of strong solutions to (P)($p, v_0$) and to derive an error estimate between the strong solution and approximate solutions constructed with the FVM.
Before stating our main results, we introduce a finite volume scheme for (P)($p, v_0$).
In the same manner as in \cite{M-A-1}, we discretize the problem (P)($p, v_0$) with respect to the space variable.
Let $n \in \mathbb{N}$ with $n \geq 3$.
We divide the interval $(0,1)$ into $n$ control volumes $\{ V_i^{(n)} \}_{i=1}^{n}$ given by
\begin{align*}
V_{i}^{(n)} \coloneqq
\begin{dcases}
[x_{i-1}^{(n)}, x_{i}^{(n)}), &\qquad i = 1, \ldots, n-1,\\
[x_{n-1}^{(n)}, x_{n}^{(n)}], &\qquad i = n,
\end{dcases}
\end{align*}
where $\Delta x^{(n)} \coloneqq 1/n$ and $x_{i}^{(n)} \coloneqq i \Delta x^{(n)}$.
For $v \colon [0, T) \times (0, 1) \to \mathbb{R}$, $t \in [0, T)$, and $i = 1, \ldots, n$, the mean value of $v(t)$ over $V_{i}^{(n)}$ is defined by
\begin{align*}
v_i^{(n)}(t) \coloneqq \frac{1}{\Delta x^{(n)}}\int_{V_i^{(n)}} v(t, x) dx.
\end{align*}
By integrating \eqref{P_eq} over $V_i^{(n)}$, we derive the following system of ordinary differential equations with respect to the time variable:
\begin{align}
h'(v_i^{(n)}) \frac{d}{dt} v_{i}^{(n)}
\label{OP_eq}
&= \begin{dcases}
\frac{1}{\Delta x^{(n)}} \left( \frac{v_2^{(n)} - v_1^{(n)}}{\Delta x^{(n)}} + b\left( \frac{v_1^{(n)} + v_2^{(n)}}{2} \right) p_1^{(n)} \right), &\qquad i = 1, \\
\frac{1}{\Delta x^{(n)}} \left( \frac{v_{i+1}^{(n)} - v_i^{(n)}}{\Delta x^{(n)}} + b\left( \frac{v_i^{(n)} + v_{i+1}^{(n)}}{2} \right) p_i^{(n)} \right. \\
\quad \left. - \left( \frac{v_i^{(n)} - v_{i-1}^{(n)}}{\Delta x^{(n)}} + b\left( \frac{v_{i-1}^{(n)} + v_{i}^{(n)}}{2} \right) p_{i-1}^{(n)} \right) \right), &\qquad i = 2, \ldots, n-1, \\
-\frac{1}{\Delta x^{(n)}} \left( \frac{v_n^{(n)} - v_{n-1}^{(n)}}{\Delta x^{(n)}} + b\left( \frac{v_{n-1}^{(n)} + v_n^{(n)}}{2} \right) p_{n-1}^{(n)} \right), &\qquad i=n
\end{dcases}
\end{align}
on $[0, T]$, where
\begin{align*}
p_i^{(n)} \coloneqq \frac{1}{\Delta x^{(n)}} \int_{V_i^{(n)}} p(x) dx.
\end{align*}
%We note that the boundary condition \eqref{P_BC} has been used for the integrations over $V_1^{(n)}$ and $V_n^{(n)}$.
%Thus, the boundary condition is naturally incorporated into the discrete scheme \eqref{OP_eq}.
The initial values of the above system are given by
\begin{align}\label{OP_IC}
v_i^{(n)}(0) 
= v_{0, i}^{(n)} 
\coloneqq \frac{1}{\Delta x^{(n)}} \int_{V_i^{(n)}} v_0(x)  dx, \qquad i = 1, \ldots, n.
\end{align}
Let (OP)$^{(n)}$($\bm{p}^{(n)}, \bm{v}_0^{(n)}$) denote the Cauchy problem \eqref{OP_eq} and \eqref{OP_IC}, where $\bm{p}^{(n)} \coloneqq (p_1^{(n)}, \ldots, p_n^{(n)})$ and $\bm{v}_0^{(n)} \coloneqq (v_{0, 1}^{(n)}, \ldots, v_{0, n}^{(n)})$.
If $p \in W^{1,2}(0, T; H)$, then $p_i^{(n)} \in C([0, T])$ for any $i = 1, \ldots, n$, which in turn implies that (OP)$^{(n)}$($\bm{p}^{(n)}, \bm{v}_0^{(n)}$) has a unique solution $\bm{v}^{(n)} = (v_1^{(n)}, \ldots, v_n^{(n)}) \in (C([0,T]))^n$.

By using $\bm{v}^{(n)}$, we define an approximate solution $v^{(n)}$ of (P)($p, v_0$) as a piecewise constant function with respect to the space variable:
\begin{align}
\label{AS}
v^{(n)}(t, x) \coloneqq \sum^{n}_{i=1}  \chi_i^{(n)}(x) v_{i}^{(n)}(t) \qquad \text{for $(t,x) \in Q(T)$},
\end{align}
where $\chi_i^{(n)}$ is the characteristic function of $V_i^{(n)}$.
Similarly, we put
\begin{align*}
p^{(n)}(t, x) \coloneqq \sum^{n}_{i=1}  \chi_i^{(n)}(x) p_{i}^{(n)}(t) \qquad \text{for $(t,x) \in Q(T)$}.
\end{align*}
We define a spatial difference and a midpoint value of $v^{(n)}$ by
\begin{align*}
\delta_x v^{(n)} (t,x)
&\coloneqq \frac{v^{(n)}(t, x) - v^{(n)}(t, x - \Delta x^{(n)})}{\Delta x^{(n)}}, \\
\overline{v}^{(n)} (t,x)
&\coloneqq \frac{v^{(n)}(t, x) + v^{(n)}(t, x - \Delta x^{(n)})}{2},
\end{align*}
respectively.
The same definitions are used for $v_{i}^{(n)}$:
\begin{align*}
\delta_x v_{i}^{(n)} (t)
\coloneqq \frac{v_{i}^{(n)}(t) - v_{i-1}^{(n)}(t)}{\Delta x^{(n)}}, \qquad
\overline{v}_{i}^{(n)} (t)
\coloneqq \frac{v_{i}^{(n)}(t) + v_{i-1}^{(n)}(t)}{2}.
\end{align*}
As approximations of $\partial_x v$ and $\partial_{x}^2 v$, we set
\begin{align*}
\widetilde{v}^{(n)} (t, x) &\coloneqq
\begin{dcases}
0, &\qquad 0 \leq x < \Delta x^{(n)},\\
\delta_x v^{(n)}(t, x), &\qquad \Delta x^{(n)} \leq x \leq 1,
\end{dcases}\\
\widehat{v}^{(n)} (t, x) &\coloneqq
\begin{dcases}
\delta_x(\delta_x v^{(n)}(t, x+\Delta x^{(n)} )), &\qquad \Delta x^{(n)} \leq x \leq 1 - \Delta x^{(n)},\\
0, &\qquad \text{otherwise}.
\end{dcases}
\end{align*}

By deriving uniform estimates for the approximate solutions constructed in \eqref{AS} and applying a compactness argument, we establish the existence of strong solutions to (P)($p, v_0$).

\begin{theorem} \label{thm_ESS}
Let $T > 0$.
Assume (A1) and (A2).
If $p \in W^{1,2}(0, T; H) \cap L^\infty(0,T; L^\infty(0,1)) \cap L^2(0,T; X)$ and $v_0 \in X$, then {\rm (P)($p, v_0$)} has a unique strong solution
\begin{align*}
v \in W^{1,2}(0,T; H) \cap L^\infty(0,T; X) \cap L^2(0,T; H^2(0,1)).
\end{align*}
\end{theorem}

\begin{remark}
The uniqueness of weak solutions to (P)($p, v_0$), defined by Definition \ref{def_WS} in the next section, has already been proved in \cite[Theorem 2.1]{M-A-1}.
See also Proposition \ref{prop_EUWS}.
Since every strong solution is a weak solution, the uniqueness of strong solutions follows from Proposition \ref{prop_EUWS}.
\end{remark}

Furthermore, we derive an error estimate between the strong solution and the approximate solution.

\begin{theorem} \label{thm_EE}
Let $n \in \mathbb{N}$ with $n \geq 3$ and let $\Delta x^{(n)} \coloneqq 1/n$.
Under the same assumption as in Theorem \ref{thm_ESS}, let $v^{(n)}$ be the approximate solution defined by \eqref{AS} and let $v$ be the strong solution of {\rm (P)($p, v_0$)} given by Theorem \ref{thm_ESS}.
Then, there exists $C = C(h, b, p, v_0, T) > 0$ such that
\begin{align*}
&|v - v^{(n)}|_{L^\infty(0, T; H)}^2 + \int^t_0 |\partial_x v(\tau, \cdot) - \widetilde{v}^{(n)}(\tau, \cdot) |^2_{H_\Delta^{(n)}}  d\tau \leq C (\Delta x^{(n)})^{\frac{1}{2}}
\end{align*}
for all $t \in [0, T]$, where
\begin{align*}
H_\Delta^{(n)} \coloneqq L^2(\Delta x^{(n)}, 1), \qquad
|\cdot|_{H_\Delta^{(n)}} \coloneqq |\cdot|_{L^2(\Delta x^{(n)}, 1)}.
\end{align*}
\end{theorem}

In the proof of Theorem \ref{thm_EE}, the Gagliardo--Nirenberg type inequality for the difference between a continuous function and a piecewise constant function plays an essential role (see Lemma \ref{lem_GNineq_dc} in Section \ref{sec_EE}).

\begin{remark}
We can show that the order of the error estimate in Theorem \ref{thm_EE} can be improved from $O((\Delta x^{(n)})^{1/4})$ to $O((\Delta x^{(n)})^{1/2})$ under stronger regularity assumptions, which will be discussed in a forthcoming paper.
\end{remark}

\begin{remark}
The strong solvability for (P)($p, v_0$) with the Dirichlet boundary condition instead of the Robin type boundary condition \eqref{P_BC} can be established by the standard theory of evolution equations without utilizing the FVM.
Furthermore, approximate solutions with error estimates similar to those in Theorem \ref{thm_EE} can be constructed based on the FVM.
It seems possible to extend our results in this paper and \cite{M-A-2} to higher-dimensional cases.
However, such an extension is not straightforward, since our approach heavily relies on the Gagliardo--Nirenberg inequality in Lemma \ref{lem_1DGNineq} below.
\end{remark}

At the end of this section, we give the following preliminary lemma, which follows from straightforward computations.

\begin{lemma} \label{lem_1DGNineq}
For any $u \in X$, it holds that
\begin{align*}
|u|^2_{L^\infty(0,1)}
&\leq |u|_H^2 + 2|u|_H |\partial_x u|_H.
%|u|^4_{L^4(0,1)} 
%&\leq (|u|_H^2 + 2|u|_H |\partial_x u|_H)^2.
\end{align*}
\end{lemma}

\section{Existence of strong solutions} \label{sec_ESS}
This section is devoted to the proof of Theorem \ref{thm_ESS}.
As a preparation, we first recall some results of Morimura--Aiki \cite{M-A-1}.
In this work, we showed that (P)($p, v_0$) has a unique weak solution, provided that $p \in L^4(0, T; H)$ and $v_0 \in H$.
Moreover, we proved that approximate solutions constructed with the FVM converge to the weak solution.

To state these results rigorously, we define a weak solution to (P)($p, v_0$).
\begin{definition} \label{def_WS}
Let $T > 0$.
A function $v \colon [0, T) \times (0, 1) \to \mathbb{R}$ is said to be a weak solution of (P)($p, v_0$) if it satisfies
\begin{align*}
&v \in  L^{\infty}(0, T; H) \cap L^2(0, T; X)
\end{align*}
and the identity
\begin{align}
\label{P_D2}
-\int_{Q(T)}  h(v) \partial_t \eta dxdt
+ \int_{Q(T)} (\partial_x v+ b(v) p) \partial_x \eta dxdt
= \int_0^1 h(v_0) \eta(0) dx 
\end{align}
for all $\eta \in W^{1, 2}(0, T; H) \cap L^2(0, T; X)$ with $\eta(T) = 0$.
\end{definition}
The existence and uniqueness of weak solutions to (P)($p, v_0$) and uniform estimates for the approximate solutions can be stated as follows:

\begin{proposition}[\cite{M-A-1}] \label{prop_EUWS}
Let $T > 0$.
Assume (A1) and (A2).
If $p \in L^{4}(0, T; H)$ and $v_0 \in H$, then {\rm (P)($p, v_0$)} has a unique weak solution. 
\end{proposition}

\begin{lemma}[\cite{M-A-1}] \label{lem_tildevL2L2esti}
Let $n \in \mathbb{N}$.
Under the same assumptions as in Proposition \ref{prop_EUWS}, let $v^{(n)}$ be an approximate solution defined by \eqref{AS}.
Then, there exists $C_1 = C_1(h, b, p, v_0, T)> 0$ such that 
\begin{align*}
\sup_{t \in [0, T]} |v^{(n)}(t)|_H^2 dx 
\leq C_1, \qquad
\int^T_0 |\widetilde{v}^{(n)}|_H^2 dt 
\leq C_1.
\end{align*}
\end{lemma}

\begin{remark}
In \cite{M-A-1}, Lemma \ref{lem_tildevL2L2esti} was proved for the finite volume scheme corresponding to (P)($\rho, v_0$), where $\rho$ denotes an approximation of $p$ with a mollifier in space and time.
Lemma \ref{lem_tildevL2L2esti} can be shown by a slight modification of the proof of \cite[Lemma 4.1]{M-A-1}.
\end{remark}

We next derive uniform estimates for $\{ \partial_t v^{(n)} \}_n$ and $\{ \widetilde{v}^{(n)} \}_n$. 
\begin{lemma} \label{lem_dtvL2L2esti}
Let $n \in \mathbb{N}$ with $n \geq 3$.
Under the same assumptions as in Theorem \ref{thm_ESS}, there exists $C_2 = C_2(h, b, p, v_0, T) > 0$ such that
\begin{align*}
\frac{C_{h, 1}}{2} \int^{T}_0 |\partial_\tau v^{(n)}(\tau)|_H^2 d\tau
+ \frac{1}{4} \sup_{t \in [0, 1]} |\widetilde{v}^{(n)}(t)|^2_H
&\leq C_2.
\end{align*}
\end{lemma}
\begin{proof}
Multiplying \eqref{OP_eq} by $\partial_t v^{(n)}$ and integrating it over $(0, 1)$, we obtain
\begin{align}
&\int^1_0 h'(v^{(n)}) |\partial_t v^{(n)}|^2 dx \nonumber \\
&= \sum_{i=1}^{n-1} \int_{V_{i+1}^{(n)}} \delta_x v_{i+1}^{(n)}  \frac{d}{dt} v_{i}^{(n)} dx - \sum_{i=2}^{n} \int_{V_{i}^{(n)}}\delta_x v_{i}^{(n)} \frac{d}{dt} v_{i}^{(n)} dx \nonumber \\
\label{eq2}
&\hspace{10mm} + \sum_{i=1}^{n-1} \int_{V_{i}^{(n)}}
b ( \overline{v}_{i+1}^{(n)} ) p_{i}^{(n)} \frac{d}{dt} v_{i}^{(n)} dx 
 - \sum_{i=2}^{n} \int_{V_{i}^{(n)}} b ( \overline{v}_{i}^{(n)} ) p_{i-1}^{(n)} \frac{d}{dt} v_{i}^{(n)} dx
\qquad \text{on $[0, T]$}.
\end{align}
For the left-hand side, we have
\begin{align*}
\int^1_0 h'(v^{(n)}) |\partial_t v^{(n)}|^2 dx
\geq C_{h, 1} |\partial_t v^{(n)}|_H^2.
\end{align*}
The right-hand side of \eqref{eq2} is estimated as
\begin{align*}
&\sum_{i=1}^{n-1} \int_{V_{i+1}^{(n)}} \delta_x v_{i+1}^{(n)} \frac{d}{dt} v_{i}^{(n)} dx - \sum_{i=2}^{n} \int_{V_{i}^{(n)}}\delta_x v_{i}^{(n)} \frac{d}{dt} v_{i}^{(n)} dx \\
&\hspace{10mm} + \sum_{i=1}^{n-1} \int_{V_{i}^{(n)}}
b ( \overline{v}_{i+1}^{(n)} ) p_{i}^{(n)} \frac{d}{dt} v_{i}^{(n)} dx 
 - \sum_{i=2}^{n} \int_{V_{i}^{(n)}} b ( \overline{v}_{i}^{(n)} ) p_{i-1}^{(n)} \frac{d}{dt} v_{i}^{(n)} dx\\
&\leq - \sum_{i=2}^{n} \int_{V_{i}^{(n)}} \delta_x v_{i}^{(n)} \delta_x \frac{d}{dt} v_{i}^{(n)} dx 
- \sum_{i=2}^{n} \int_{V_{i}^{(n)}} b(\overline{v}_{i}^{(n)}) p_{i-1}^{(n)} \delta_x \frac{d}{dt} v_{i}^{(n)} dx\\
&= - \frac{1}{2} \sum_{i=2}^{n} \int_{V_{i}^{(n)}} \frac{d}{dt} (\delta_x v_{i}^{(n)})^2 dx
-\sum_{i=2}^{n} \int_{V_{i}^{(n)}} \frac{d}{dt} (b(\overline{v}_{i}^{(n)}) p_{i-1}^{(n)} \delta_x v_{i}^{(n)}) dx\\
&\hspace{10mm} + \sum_{i=2}^{n} \int_{V_{i}^{(n)}} \frac{d}{dt} (b(\overline{v}_{i}^{(n)}) p_{i-1}^{(n)}) \delta_x v_{i}^{(n)} dx\\
&= - \frac{1}{2} \frac{d}{dt} |\widetilde{v}^{(n)}|_{H_\Delta^{(n)}}^2
- \frac{d}{dt} \int^1_{\Delta x^{(n)}} b( \overline{v}^{(n)}(\cdot, x) ) p^{(n)}(\cdot, , x-\Delta x^{(n)}) \delta_x v^{(n)}(\cdot, x) dx\\
&\hspace{10mm} +\int^1_{\Delta x^{(n)}} b'(\overline{v}^{(n)}(\cdot, x)) \partial_t\overline{v}^{(n)}(\cdot, x) p^{(n)}(\cdot, x-\Delta x^{(n)}) \delta_x v^{(n)}(\cdot, x) dx\\
&\hspace{10mm} +\int^1_{\Delta x^{(n)}} b(\overline{v}^{(n)}(\cdot, x)) \partial_t p^{(n)}(\cdot, x-\Delta x^{(n)}) \delta_x v^{(n)} (\cdot, x) dx.
\end{align*} 
Young's inequality yields
\begin{align*}
&\int^1_{\Delta x^{(n)}} b'(\overline{v}_{n}(\cdot, x)) \partial_t \overline{v}^{(n)} (\cdot, x) p^{(n)} (\cdot, x - \Delta x^{(n)}) \delta_x v^{(n)}(\cdot, x) dx\\
&\leq C_{b, 2} |p|_{L^\infty(Q(T))} \int^1_{\Delta x^{(n)}} | \partial_t \overline{v}^{(n)}(\cdot, x) | |\widetilde{v}^{(n)}(\cdot, x) | dx\\
&\leq \frac{C_{h, 1}}{2} \int^1_{\Delta x^{(n)}} | \partial_t\overline{v}^{(n)}(\cdot, x) |^2 dx
+ \frac{C_{b, 2}^2 |p|_{L^\infty(Q(T))}^2}{2C_{h, 1}} |\widetilde{v}^{(n)}|_{H_\Delta^{(n)}}^2\\
&\leq C_{h, 1} \int^1_{\Delta x^{(n)}} \left( \left|\frac{\partial_t v^{(n)}(\cdot, x)}{2}\right|^2 
+ \left| \frac{\partial_t v^{(n)}(\cdot, x-\Delta x^{(n)})}{2} \right|^2 \right)dx
+ \frac{C_{b, 2}^2 |p|_{L^\infty(Q(T))}^2}{2C_{h, 1}} |\widetilde{v}^{(n)}|_{H_\Delta^{(n)}}^2\\
&\leq \frac{C_{h, 1}}{2} |\partial_t v^{(n)}|_H^2
+ \frac{C_{b, 2}^2 |p|_{L^\infty(Q(T))}^2}{2C_{h, 1}}|\widetilde{v}^{(n)}|_{H_\Delta^{(n)}}^2, \\
&\int^1_{\Delta x^{(n)}} b(\overline{v}^{(n)}(\cdot, x)) \partial_t p^{(n)}(\cdot, x - \Delta x^{(n)}) \delta_x v^{(n)}(\cdot, x) dx
\leq \frac{C_{b, 2}^2}{2} | \partial_t p^{(n)} |_{L^2(0, 1-\Delta^{(n)})}^2 
+ \frac{1}{2} |\widetilde{v}^{(n)}|_{H_\Delta^{(n)}}^2.
\end{align*}
From these inequalities, it follows that
\begin{align*}
&\frac{C_{h, 1}}{2} |\partial_t v^{(n)}(t)|_H^2
+ \frac{1}{2} \frac{d}{dt} |\widetilde{v}^{(n)} (t)|_{H_\Delta^{(n)}}^2
\leq \frac{C_3}{2} |\widetilde{v}^{(n)}(t)|_{H_\Delta^{(n)}}^2
+\frac{C_{b, 2}^2}{2} | \partial_t p^{(n)}(t) |_{L^2(0, 1-\Delta^{(n)})}^2
 - \frac{d}{dt} F^{(n)}(t),
\end{align*}
where
\begin{align*}
C_3
\coloneqq \frac{C_{b, 2}^2 |p|_{L^\infty(Q(T))}^2}{C_{h, 1}} + 1, \qquad
F^{(n)} (t)
\coloneqq \int^1_{\Delta x^{(n)}} b ( \overline{v}^{(n)}(t, x) ) p^{(n)}(t, x - \Delta x^{(n)}) \widetilde{v}^{(n)}(t, x) dx.
\end{align*}
Applying Gronwall's inequality gives
\begin{align*}
&\frac{C_{h, 1}}{2} \int^t_0 |\partial_\tau v^{(n)}(\tau)|_H^2 d\tau
+ \frac{1}{2} |\widetilde{v}^{(n)}(t)|_H^2\\
&\hspace{10mm}
\leq \exp(C_3T) \left\{ \frac{C_{b, 2}^2}{2} \int^T_0 \int^1_{\Delta x^{(n)}} | \partial_\tau p^{(n)}(\tau, x - \Delta x^{(n)}) |^2 dx d\tau\right. \\
&\hspace{20mm}
\left. -\int^t_0 \exp(-C_3\tau) \frac{d}{d\tau} F^{(n)}(\tau) d\tau
 + \frac{1}{2} |\widetilde{v}^{(n)}(0)|^2_H) \right\}
\qquad \text{for any $t \in [0, T]$}.
\end{align*}
By Lemma \ref{lem_tildevL2L2esti} and Young's inequality, we see that
\begin{align*}
&- \int^t_0 \exp(-C_3\tau) \frac{d}{d\tau} F^{(n)}(\tau) d\tau\\
&= -C_3 \int^t_0 \exp(-C_3\tau) F^{(n)}(\tau) d\tau 
- \exp(-C_3t)F^{(n)}(t) + F(0)\\
&\leq C_3 C_{b, 2} |p|_{L^\infty(Q(T))} \int^t_0 \int^1_0 |\widetilde{v}^{(n)}(\tau, x)| dx d\tau
+ C_{b, 2} |p|_{L^\infty(Q(T))} \left( | \widetilde{v}^{(n)}(t)| _{H} + | \widetilde{v}^{(n)}(0)|_{H} \right)\\
&\leq C_3 C_{b, 2} |p|_{L^\infty(Q(T))} \sqrt{T C_1} + \frac{1}{4}|\widetilde{v}^{(n)}(t)|_H^2
+ C_{b, 2}^2 |p|_{L^\infty(Q(T))}^2 + \frac{C_{b, 2} |p|_{L^\infty(Q(T))}}{2} (|\widetilde{v}^{(n)}(0)|_H^2 + 1).
\end{align*}
Thus, we get
\begin{align}
&\frac{C_{h, 1}}{2} \int^t_0 |\partial_t v^{(n)}(\tau)|_H^2 d\tau
+ \frac{1}{4} |\widetilde{v}^{(n)}(t)|_H^2 \nonumber\\
&\leq \exp(C_3T) \left\{ \frac{C_{b, 2}^2}{2} \int^T_0 | \partial_\tau p^{(n)} |_H^2 d\tau + C_3 C_{b, 2} |p|_{L^\infty(Q(T))} \sqrt{T C_1}  \right. \nonumber\\ 
&\label{esti1}
\hspace{10mm}\left. + C_{b, 2}^2 |p|_{L^\infty(Q(T))}^2 + \frac{C_{b, 2} |p|_{L^\infty(Q(T))}}{2} (|\widetilde{v}^{(n)}(0)|_H^2 + 1) 
 + \frac{1}{2} |\widetilde{v}^{(n)}(0)|^2_H) \right\} 
\end{align}
for any $t \in [0, T]$. We remark that
\begin{align}
\int^T_0 | \partial_\tau p^{(n)} (\tau) |_H^2 d\tau
&= \frac{1}{\Delta x^{(n)}} \sum_{i = 1}^n \int^T_0 \int_{V_{i}^{(n)}} \left| \frac{d}{d\tau} \int_{V_{i}^{(n)}} p(\tau, \xi) d\xi \right|^2 dx d\tau \nonumber\\
&= \frac{1}{\Delta x^{(n)}} \sum_{i = 1}^n \int^T_0 \int_{V_{i}^{(n)}} \left| \int_{V_{i}^{(n)}} \partial_\tau p(\tau, \xi) d\xi \right|^2 dx d\tau \nonumber\\
&\leq \sum_{i = 1}^n \int_{V_{i}^{(n)}} dx \int^T_0 \int_{V_{i}^{(n)}} |\partial_\tau p(\tau, \xi)|^2 d\xi d\tau \nonumber\\
&\label{esti7} \leq \int^T_0 |\partial_\tau p(\tau)|_H^2 dt.
\end{align}
The definition of $v_{0, i}^{(n)}$ implies
\begin{align*} 
|\widetilde{v}^{(n)}(0)|_H^2
&= \int^1_{\Delta x^{(n)}} \left| \frac{1}{\Delta x} \sum^n_{i=2} \chi_i^{(n)}(x) (v_{0, i}^{(n)} -  v_{0, i-1}^{(n)}) \right|^2 dx\\
&= \frac{1}{(\Delta x^{(n)})^4} \sum^n_{i=2} \int_{V_{i}^{(n)}} \left| \int_{V_{i}^{(n)}} v_0(\xi) d\xi -  \int_{V_{i-1}^{(n)}} v_0(\xi) d\xi \right|^2 dx\\
%&= \frac{1}{(\Delta x^{(n)})^4} \int^1_{\Delta x^{(n)}} \left| \sum^n_{i=2} \chi_i(x) \int_{V_{i}^{(n)}} (v_0(\xi) d\xi - v_0(\xi - \Delta x^{(n)})) d\xi  \right|^2 dx \nonumber\\
%&= \frac{1}{(\Delta x^{(n)})^4} \int^1_{\Delta x^{(n)}} \left| \sum^n_{i=2} \chi_i(x) \int_{V_{i}^{(n)}} \int^\xi_{\xi - \Delta x^{(n)}} v_{0y}(y) dy d\xi  \right|^2 dx \nonumber\\
&= \frac{1}{(\Delta x^{(n)})^3} \sum^n_{i=2} \left| \int_{V_{i}^{(n)}} \int^\xi_{\xi - \Delta x^{(n)}} \frac{d}{dy} v_{0}(y) dy d\xi  \right|^2 \\ 
%&\leq \frac{1}{\Delta x^{(n)}} \sum^n_{i=2} \int_{V_{i}^{(n)}} \int^\xi_{\xi - \Delta x^{(n)}} \left| v_{0y}(y) \right|^2 dy d\xi \nonumber\\  
%&\leq \frac{1}{\Delta x^{(n)}} \sum^n_{i=2} \int_{V_{i}^{(n)}} \int^{\Delta x^{(n)}}_0 \left| v_{0x}(\xi - x)\right|^2 dx d\xi \nonumber\\
&\leq \frac{1}{\Delta x^{(n)}} \sum^n_{i=2} \int_{V_{i}^{(n)}} \int^{x_{i}^{(n)}}_{x_{i-1}^{(n)} - \Delta x^{(n)}} \left| \frac{d}{dy} v_{0}(y)\right|^2 dy dx \\ 
%&\leq \frac{1}{\Delta x^{(n)}} \int^{\Delta x^{(n)}}_0 \int^{1 - x}_{\Delta x^{(n)} - x} \left| v_{0y}(y) \right|^2 dy dx \nonumber\\ 
&\leq \left| \frac{d}{dx} v_{0} \right|_H^2.
\end{align*}
Hence, by setting the right-hand side of \eqref{esti1} to $C_2$, we conclude that Lemma \ref{lem_dtvL2L2esti} holds.
\end{proof}

We also use an uniform estimate for $\{ \widehat{v}^{(n)} \}_{n}$.
\begin{lemma} \label{lem_dxxvL2L2esti}
Let $n \in \mathbb{N}$ with $n \geq 3$.
Under the same assumption as in Theorem \ref{thm_ESS}, there exists $C_4 = C_4(h, b, p, v_0, T) > 0$ such that
\begin{align*}
\int^T_0 |\widehat{v}^{(n)}(t)|_H^2 dt \leq C_4.
\end{align*}
\end{lemma}
\begin{proof}
We transform \eqref{OP_eq} as follows:
\begin{align}
\label{eq1}
\delta_x (\delta_x v_{i+1}^{(n)}) = h'(v_{i}^{(n)}) \frac{d}{dt} v_{i}^{(n)} - \frac{1}{\Delta x^{(n)}} (b(\overline{v}_{i+1}^{(n)})p_i^{(n)}) - b(\overline{v}_{i}^{(n)})p_{i-1}^{(n)})
\end{align}
on $[0, T]$ and for $i = 2, \ldots, n-1$.
Squaring \eqref{eq1} and adding them with respect to $i = 2, \ldots, n-1$ yield
\begin{align*}
&\sum^{n-1}_{i=2} |\delta_x (\delta_x v_{i+1}^{(n)})|^2\\
&= \sum^{n-1}_{i=2} \left| h'(v_{i}^{(n)}) \frac{d}{dt} v_{i}^{(n)} - \frac{1}{\Delta x^{(n)}} (b(\overline{v}_{i+1}^{(n)}) p_i^{(n)}) - b(\overline{v}_{i}^{(n)}) p_{i-1}^{(n)}) \right|^2\\
&\leq 2C_{h, 2} \sum^{n-1}_{i=2} \left| \frac{d}{dt} v_{i}^{(n)} \right|^2 + \frac{2}{(\Delta x^{(n)})^2} \sum^{n-1}_{i=2} | b(\overline{v}_{i+1}^{(n)})p_i^{(n)} - b(\overline{v}_{i}^{(n)}) p_{i-1}^{(n)} |^2
\end{align*}
on $[0, T]$.
For the second term on the right-hand side, thanks to the Lipschitz continuity of $b$, we have
\begin{align*}
&\frac{2}{(\Delta x^{(n)})^2} \sum^{n-1}_{i=2} | b(\overline{v}_{i+1}^{(n)})p_{i}^{(n)} - b(\overline{v}_{i}^{(n)}) p_{i-1}^{(n)} |^2\\
&\leq \frac{18}{(\Delta x^{(n)})^2} \sum^{n-1}_{i=2} (|(b(\overline{v}_{i+1}^{(n)}) - b(v_i^{(n)})) p_{i}^{(n)}|^2 \\
&\hspace{10mm}+ |b(v_i^{(n)})(p_{i}^{(n)} - p_{i-1}^{(n)})|^2
+ |(b(v_i^{(n)}) - b(\overline{v}_{i}^{(n)})) p_{i-1}^{(n)}|^2)\\
&\leq 9C_{b, 2}^2 |p|_{L^\infty(Q(T))}^2 \sum^{n}_{i=2} |\delta_x v_i^{(n)}|^2
+ 18C_{b, 2}^2 \sum^{n-1}_{i=2} |\delta_x p_{i}^{(n)}|^2.
\end{align*}
Thus, we get
\begin{align*}
|\widehat{v}^{(n)}|_H^2
&\leq 2C_{h, 2} |\partial_t v^{(n)}|_H^2
+ 9C_{b, 2}^2 |p|_{L^\infty(Q(T))}^2 |\widetilde{v}^{(n)}|_H^2
+ 18C_{b, 2}^2 |\delta_x p^{(n)}|_{H_\Delta^{(n)}}^2
\end{align*}
on $[0, T]$, which implies
\begin{align*}
\int^T_0 |\widehat{v}^{(n)}|_H^2 dt
&\leq 2C_{h, 2} \int^T_0 |\partial_t v^{(n)}|_H^2 dt
+ 9C_{b, 2}^2 |p|_{L^\infty(Q(T))} \int^T_0 |\widetilde{v}^{(n)}|_H^2 dt
+ 18C_{b, 2}^2 \int^T_0 |\delta_x p^{(n)}|^2_{H_\Delta^{(n)}} dt.
\end{align*}
Similarly to \eqref{esti7}, we obtain
\begin{align*}
\int^T_0 |\delta_x p^{(n)}|_{H_\Delta^{(n)}}^2 dt
\leq \int^T_0 |\partial_x p|_H^2 dt.
\end{align*}
Hence, by setting
\begin{align*}
C_4 
\coloneqq  \frac{4C_2C_{h, 2}}{C_{h, 1}}
+ 36C_{b, 2}^2 |p|_{L^\infty(Q(T))} C_2
+ 18C_{b, 2}^2 \int^T_0 |\partial_x p|_H^2 dt,
\end{align*}
we arrive at the desired conclusion.
\end{proof}

We are now ready to prove Theorem \ref{thm_ESS}.

\begin{proof}[Proof of Theorem \ref{thm_ESS}]
From Lemmas \ref{lem_dtvL2L2esti} and \ref{lem_dxxvL2L2esti}, it follows that
\begin{itemize}
\item $\{ \partial_t v^{(n)} \}_n$ and $\{ \widehat{v}^{(n)} \}_n$ are bounded in $L^2(0,T; H)$;
\item $\{ \widetilde{v}^{(n)} \}_n$ is bounded in $L^\infty(0,T; H)$.
\end{itemize}
Hence, there exists a subsequence $\{ v^{(n_j)} \}_j$ of $\{ v^{n} \}_n$ such that
\begin{alignat*}{2}
\partial_t v^{(n_j)} &\to \partial_t v, \quad \widehat{v}^{(n_j)} \to \partial_{x}^2 v
&&\qquad \text{weakly in $L^2(0, T; H)$},\\
\widetilde{v}^{(n_j)} &\to \partial_x v
&&\qquad \text{weakly$\ast$ in $L^\infty(0, T; H)$}
\end{alignat*}
as $j \to \infty$, where $v$ is a unique weak solution of (P)($p, v_0$) satisfying
\begin{align*}
v \in W^{1,2}(0,T; H) \cap L^\infty(0,T; X) \cap L^2(0,T; H^2(0,T)).
\end{align*}
Moreover, we obtain
\begin{align*}
-\int_{Q(T)}  \partial_t (h(v)) \eta  dxdt + \int_{Q(T)} \partial_x (\partial_x v + b(v) p) \eta dxdt = 0
\end{align*}
for any $\eta \in C_0^\infty(Q(T))$.
As a consequence, \eqref{P_eq} holds almost everywhere in $Q(T)$.
In addition, the boundary condition \eqref{P_BC} and the initial condition \eqref{P_IC} follow from \eqref{P_D2}.
\end{proof}

\section{Error estimate} \label{sec_EE}
In this section, we give the proof of Theorem \ref{thm_EE}.
For this purpose, we suppose the same condition as in Theorem \ref{thm_ESS} throughout this section.
The main argument of this section is to estimate the difference between the approximate solution and the strong solution to (P)($p, v_0$) in $L^\infty(0, T; H)$. 

We first provide the Gagliardo--Nirenberg type inequality for the difference between a continuous function and a piecewise constant function. 

\begin{lemma} \label{lem_GNineq_dc}
Let $n \in \mathbb{N}$ with $n \geq 2$ and define
\begin{align*}
S^{(n)} \coloneqq \left\{ s^{(n)} = \sum_{i=1}^n z_{i} \chi_i^{(n)} \middle| {\ } \text{$z_i \in \mathbb{R}$ for every $i = 1, \ldots, n$}.\right\}.
\end{align*}
Then, there exists $C_5 > 0$, independent of $n$, such that
\begin{align*}
\int^1_{\Delta x^{(n)}} |s^{(n)} - w|^4 dx
&\leq C_5 |s^{(n)} - w|_{H_\Delta^{(n)}}^4
+ C_5 ( |\delta_x s^{(n)} - \partial_x w|_{H_\Delta^{(n)}}^2\\
&\hspace{10mm}+ |\partial_x w- \partial_x w(\cdot - \Delta x^{(n)})|_{H_\Delta^{(n)}}^2
+ \Delta x^{(n)} |\partial_x w|_H^2 ) |s^{(n)} - w|_{H_\Delta^{(n)}}^2
\end{align*}
holds for all $w \in X$ and $s^{(n)} \in S^{(n)}$.
\end{lemma}
\begin{proof}
Let $w \in X$, $s^{(n)} \in S^{(n)}$, and $x, y \in [\Delta x^{(n)}, 1]$.
If $x_{i}^{(n)} > x > y \geq x_{i-1}^{(n)}$ for some $i = 2, \ldots, n$, then we obtain
\begin{align*}
|s^{(n)}(x) - w(x)|^2 - |s^{(n)}(y) - w(y)|^2
&= |s^{(n)}(x_{i-1}^{(n)}) - w(x)|^2 - |s^{(n)}(x_{i-1}^{(n)}) - w(y)|^2\\
&= |w(y) - w(x)| |2s^{(n)}(x_{i-1}^{(n)}) - w(x) - w(y)|\\
&\leq \int_{V_{i}^{(n)}} |\partial_\xi w(\xi)| d\xi (|s^{(n)}(x) - w(x)| + |s^{(n)}(y) - w(y)|).
\end{align*}
If $x_{i}^{(n)} > x \geq x_{i-1}^{(n)}$ and $x_{j}^{(n)} > y \geq x_{j-1}^{(n)}$ where $n \geq i > j \geq 2$, then it follows that
\begin{align*}
&|s^{(n)}(x) - w(x)|^2 - |s^{(n)}(y) - w(y)|^2\\
&= |s^{(n)}(x_{i-1}^{(n)}) - w(x)|^2 - |s^{(n)}(x_{j-1}^{(n)}) - w(y)|^2\\
&= |s^{(n)}(x_{i-1}^{(n)}) - s^{(n)}(x_{j-1}^{(n)}) - w(x) + w(y)||s^{(n)}(x_{i-1}^{(n)}) + s^{(n)}(x_{j-1}^{(n)}) - w(x) - w(y)|.
\end{align*}
A part of the right-hand side is estimated as
\begin{align*}
&|s^{(n)}(x_{i-1}^{(n)}) - s^{(n)}(x_{j-1}^{(n)}) - w(x) + w(y)|\\
&\leq \left| \sum^{i-1}_{k = j} (s^{(n)}(x_k) - s^{(n)}(x_{k-1})) - \int^{x_{i-1}^{(n)}}_{x_{j-1}^{(n)}} \partial_\xi w(\xi) d\xi \right|
+ \left| -\int^x_{x_{i-1}^{(n)}} \partial_\xi w(\xi) d\xi + \int^y_{x_{j-1}^{(n)}} \partial_\xi w(\xi) d\xi \right|\\
&\leq \sum^{i-1}_{k = j} \int_{V_{k}^{(n)}} \left| \frac{s^{(n)}(x_k) - s^{(n)}(x_{k-1})}{\Delta x^{(n)}} - \partial_\xi w(\xi) \right| d\xi
+ \int^x_{x_{i-1}^{(n)}} \left|\partial_\xi w(\xi)\right| d\xi + \int^y_{x_{j-1}^{(n)}} \left|\partial_\xi w(\xi)\right| d\xi\\
&\leq \int^{x_{i-1}^{(n)}}_{x_{j-1}^{(n)}} \left| \delta_x s^{(n)}(\xi + \Delta x^{(n)}) - \partial_\xi w(\xi) \right| d\xi
%&\leq \int^{x_{i-1}^{(n)}}_{x_{j-1}^{(n)}} \left| \frac{s^{(n)}(\xi + \Delta x^{(n)}) - s^{(n)}(\xi))}{\Delta x^{(n)}} - \partial_\xi w(\xi) \right| d\xi\\
+ \int_{V_{i}^{(n)}} \left|\partial_\xi w(\xi)\right| d\xi + \int_{V_{j}^{(n)}} \left|\partial_\xi w(\xi)\right| d\xi\\
&\leq \int^{x_{i}^{(n)}}_{x_{j}^{(n)}} \left| \delta_x s^{(n)}(\xi) - \partial_\xi w(\xi) \right| d\xi
%&\leq \int^{x_{i}^{(n)}}_{x_{j}^{(n)}} \left| \frac{s^{(n)}(\xi) - s^{(n)}(\xi - \Delta x^{(n)}))}{\Delta x^{(n)}} - \partial_\xi w(\xi) \right| d\xi \\
 + \int^{x_{i}^{(n)}}_{x_{j}^{(n)}} | \partial_\xi w(\xi) - \partial_\xi w(\xi - \Delta x^{(n)}) | d\xi \\
&\hspace{10mm} + \int_{V_{i}^{(n)}} \left|\partial_\xi w(\xi)\right| d\xi + \int_{V_{j}^{(n)}} \left|\partial_\xi w(\xi)\right| d\xi.
\end{align*}
Thus, we get
\begin{align*}
|s^{(n)}(x) - w(x)|^2 - |s^{(n)}(y) - w(y)|^2
\leq I_{1, i, j}  (|s^{(n)}(x) - w(x)| + |s^{(n)}(y)  - w(y)|)
\end{align*}
for $x_{j-1} \leq y < x_{j} \leq x_{i-1} \leq x < x_{i}$, where
\begin{align*}
I_{1, i, j} 
&\coloneqq \int^{x_{i}^{(n)}}_{x_{j}^{(n)}} | \delta_x s^{(n)}(\xi) - \partial_\xi w(\xi) | d\xi 
+ \int^{x_{i}^{(n)}}_{x_{j}^{(n)}} | \partial_\xi w(\xi) - \partial_\xi w(\xi - \Delta x^{(n)})| d\xi \\
&\hspace{10mm}+ \int_{V_{i}^{(n)}} \left|\partial_\xi w(\xi)\right| d\xi 
+ \int_{V_{j}^{(n)}} \left|\partial_\xi w(\xi)\right| d\xi.
\end{align*}
Hence, for any $x, y \in [\Delta x^{(n)}, 1]$, we have
\begin{align}
\label{esti4}
|s^{(n)}(x) - w(x)|^2
\leq |s^{(n)}(y) - w(y)|^2 + I_2 (|s^{(n)}(x) - w(x)| + |s^{(n)}(y)  - w(y)|),
\end{align}
where
\begin{align*}
I_2 &\coloneqq \left( \int^1_{\Delta x^{(n)}} | \delta_x s^{(n)}(\xi) - \partial_\xi w(\xi)|^2 d\xi \right)^\frac{1}{2}\\
&\hspace{10mm}+ \left(\int^1_{\Delta x^{(n)}} | \partial_\xi w(\xi) - \partial_\xi w(\xi - \Delta x^{(n)}) |^2 d\xi \right)^\frac{1}{2} 
 + 2(\Delta x^{(n)})^\frac{1}{2}|\partial_x w|_H.
\end{align*}
Integration of \eqref{esti4} with respect to $y$ over $(\Delta x^{(n)}, 1)$ gives
\begin{align}
&\frac{1}{2} |s^{(n)}(x) - w(x)|^2 \nonumber \\
&\leq (1 - \Delta x^{(n)}) |s^{(n)}(x) - w(x)|^2 \nonumber \\
\label{esti5}
&\leq |s^{(n)} - w|_{H_\Delta^{(n)}}^2 
 + I_2 ( |s^{(n)}(x) - w(x)| + |s^{(n)} - w|_{H_\Delta^{(n)}})
\end{align}
for any $x \in [\Delta x^{(n)}, 1]$. 
We note that the first inequality in \eqref{esti5} follows from the fact that $1 - \Delta x^{(n)} \geq 1/2$.
Squaring \eqref{esti5} and integrating it with respect to $x$ over $(\Delta x^{(n)}, 1)$, we obtain
\begin{align*}
&\int^1_{\Delta x^{(n)}} |s^{(n)} - w|^4 dx
\leq 8 |s^{(n)} - w|_{H_\Delta^{(n)}}^4
 + 32I_2^2 |s^{(n)} - w|_{H_\Delta^{(n)}}^2.
\end{align*}
Consequently, by setting $C_5 \coloneqq 1152$, we arrive at the conclusion of Lemma \ref{lem_GNineq_dc}.
\end{proof}

\begin{lemma} \label{lem_LinfLinfesti}
Let $n \in \mathbb{N}$ with $n \geq 3$ and let $v^{(n)}$ be the approximate solution defined by \eqref{AS}.
Then, there exists $C_{6} = C_{6}(h, b, p, v_0, T) > 0$ such that
\begin{align*}
| v^{(n)} |_{L^\infty(Q(T))}
&\leq C_{6}.
\end{align*}
\end{lemma}
\begin{proof}
Let $i, j = 1, \ldots, n$.
%If $i < j$, then we have
%\begin{align}
%|v_i^{(n)}|
%&\leq \sum^{j -1}_{k = i} |v_{k+1}^{(n)} - v_{k}^{(n)}| + |v_{j}^{(n)}| \nonumber \\
%\label {esti8}
%&\leq \Delta x^{(n)} \sum^n_{k = 2} |\delta_x v_k^{(n)}| + |v_{j}^{(n)}|
%\end{align}
%on $[0, T]$.
%The second inequality in \eqref{esti8} still holds for $i \geq j$.
By a simple computation, we have
\begin{align}
\label{esti8}
|v_i^{(n)}|
\leq \Delta x^{(n)} \sum^n_{k = 2} |\delta_x v_k^{(n)}| + |v_{j}^{(n)}|
\end{align}
on $[0, T]$.
Taking a sum of \eqref{esti8} with respect to $j$ gives
\begin{align*}
n |v_{i}^{(n)}| 
&\leq n \Delta x^{(n)} \sum_{k=2}^{n} |\delta_x v_k^{(n)}|
+ \sum_{j = 1}^{n} |v_{j}^{(n)}|\\
&\leq n |\widetilde{v}^{(n)}|_H + n |v^{(n)}|_H.
\end{align*}
Thus, by Lemmas \ref{lem_tildevL2L2esti} and \ref{lem_dtvL2L2esti}, we get
\begin{align*}
|v_i^{(n)}| \leq |\widetilde{v}^{(n)}|_H + |v^{(n)}|_H,
\end{align*}
which implies the desired estimate.
\end{proof}

We are now in a position to prove Theorem \ref{thm_EE}.

\begin{proof}[Proof of Theorem \ref{thm_EE}]
Multiplying the difference between \eqref{P_eq} and \eqref{OP_eq} by $v - v^{(n)}$ and integrating both sides over $(0, 1)$, we have
\begin{align*}
&\int^1_0 (\partial_t (h(v) - h(v^{(n)}))) (v - v^{(n)}) dx\\
&= \int^1_0 \partial_x (\partial_x v + b(v)p)(v - v^{(n)}) dx
- \frac{1}{\Delta x^{(n)}}\sum^{n-1}_{i = 1} \int_{V_{i}^{(n)}} \delta_x v_{i+1}^{(n)} (v - v^{(n)}) dx\\
&\hspace{10mm} + \frac{1}{\Delta x^{(n)}}\sum^{n}_{i = 2} \int_{V_{i}^{(n)}} \delta_x v_i^{(n)} (v - v^{(n)}) dx
- \frac{1}{\Delta x^{(n)}}\sum^{n-1}_{i = 1} \int_{V_{i}^{(n)}}  b(\overline{v}_{i+1}^{(n)}) p_i^{(n)} (v - v^{(n)}) dx\\
&\hspace{10mm} + \frac{1}{\Delta x^{(n)}}\sum^{n}_{i = 2} \int_{V_{i}^{(n)}} b(\overline{v}_i^{(n)}) p_{i-1}^{(n)} (v - v^{(n)}) dx\\
&=: \sum^{7}_{j = 3} I_j
\end{align*}
a.e. on $[0, T]$.
For the left-hand side, we obtain
\begin{align*}
&\int^1_0 (\partial_t (h(v) - h(v^{(n)}))) (v - v^{(n)}) dx\\
&= \int^1_0 \partial_t v (h'(v) - h'(v^{(n)}))) (v - v^{(n)}) dx
+ \int^1_0 h'(v^{(n)}) \frac{1}{2}\partial_t (v - v^{(n)})^2 dx\\
&= \int^1_0 \partial_t v (h'(v) - h'(v^{(n)}))) (v - v^{(n)}) dx
+ \frac{d}{dt} \int^1_0 h'(v^{(n)}) \frac{1}{2} |v-v^{(n)}|^2 dx\\
&\hspace{10mm} - \frac{1}{2} \int^1_0 (\partial_t h'(v^{(n)}))|v-v^{(n)}|^2 dx.
\end{align*}
From integration by parts, the continuity of $v^{(n)}$ and $\widetilde{v}^{(n)}$ on $V_i^{(n)}$, and the boundary condition, it follows that
\begin{align*}
I_{3}
&= \sum^n_{i=1} \int_{V_{i}^{(n)}} \partial_x(\partial_x v + b(v)p)(v - v^{(n)}) dx\\
&= -\sum^n_{i=1} \int_{V_{i}^{(n)}} (\partial_x v + b(v)p) \partial_xv dx
+ \sum^{n-1}_{i=1}(\partial_x v(\cdot, x_{i}^{(n)}) + b(v(\cdot, x_{i}^{(n)})) p(\cdot,x_{i}) ) (v(\cdot, x_{i}^{(n)}) - v^{(n)}(\cdot, x_{i-1}^{(n)})) \\
&\hspace{10mm}  - \sum^{n}_{i=2}(\partial_x v(\cdot, x_{i-1}^{(n)}) + b(v(\cdot, x_{i-1}^{(n)})) p(\cdot, x_{i-1}^{(n)})) (v(\cdot, x_{i-1}^{(n)}) - v^{(n)}(\cdot, x_{i-1}^{(n)})) \\
&= -\int^1_0 (\partial_x v + b(v)p) \partial_xv dx
+\sum^{n-1}_{i=1} (\partial_x v(\cdot, x_{i}^{(n)}) + b(v(\cdot, x_{i}^{(n)})) p(\cdot, x_{i}^{(n)})) (v^{(n)}(\cdot, x_{i}^{(n)})- v^{(n)}(\cdot, x_{i-1}^{(n)}))\\
&= -\int^1_0 (\partial_x v + b(v)p) \partial_xv dx
+\Delta x^{(n)}\sum^{n-1}_{i=1} (\partial_x v(\cdot, x_{i}^{(n)}) + b(v(\cdot, x_{i}^{(n)})) p(\cdot, x_{i}^{(n)})) \delta_x v_{i+1}^{(n)}\\
&= -\int^1_0 (\partial_x v + b(v)p) \partial_xv dx
+\sum^{n}_{i=2} \int_{V_{i}^{(n)}}(\partial_x v(\cdot, x_{i-1}^{(n)}) + b(v(\cdot, x_{i-1}^{(n)}))p(\cdot, x_{i-1}^{(n)})) \widetilde{v}^{(n)}(\cdot, x) dx\\
&= -\int^1_0 (\partial_x v + b(v)p) \partial_xv dx
+ \int^1_{\Delta x^{(n)}} (\partial_x v (\cdot, x) ) \widetilde{v}^{(n)}(\cdot, x) dx\\
&\hspace{10mm} + \sum^n_{i=2} \int_{V_{i}^{(n)}} (\partial_x v(\cdot, x_{i-1}^{(n)}) - \partial_x v + b(v(\cdot, x_{i-1}^{(n)})) p(\cdot, x_{i-1}^{(n)})) \widetilde{v}^{(n)}(\cdot, x) dx.
\end{align*}
Shifting the indices implies
\begin{align*}
I_{4} + I_{5}
&= -\frac{1}{\Delta x^{(n)}} \sum^{n-1}_{i=1} \delta_x v_{i+1}^{(n)} \int_{V_{i}^{(n)}} v dx
+\frac{1}{\Delta x^{(n)}} \sum^{n}_{i=2} \delta_x v_{i}^{(n)} \int_{V_{i}^{(n)}} v dx\\
&\hspace{10mm} + \sum^{n-1}_{i=1} \delta_x v_{i+1}^{(n)} v_{i}^{(n)}
- \sum^{n}_{i=2} \delta_x v_{i}^{(n)} v_{i}^{(n)}\\
&= \sum^{n}_{i=2} \delta_x v_{i}^{(n)} \int_{V_{i}^{(n)}} \frac{v(\cdot, x) - v(\cdot, x-\Delta x^{(n)})}{\Delta x^{(n)}} dx
- \Delta x^{(n)} \sum^{n}_{i=2} (\delta_x v_{i}^{(n)})^2\\
&= \int^1_{\Delta x^{(n)}} \widetilde{v}^{(n)} \frac{v(\cdot, x) - v(\cdot, x-\Delta x^{(n)})}{\Delta x^{(n)}} dx
- |\widetilde{v}^{(n)}|_{H_\Delta^{(n)}}^2\\
&= \int^1_{\Delta x^{(n)}} (\partial_x v) \widetilde{v}^{(n)} dx
+ \int^1_{\Delta x^{(n)}} \widetilde{v}^{(n)} \left( \frac{v(\cdot, x) - v(\cdot, x-\Delta x^{(n)})}{\Delta x^{(n)}} - \partial_x v \right)dx
- |\widetilde{v}^{(n)}|_{H_\Delta^{(n)}}^2.
\end{align*}
Similarly, we get
\begin{align*}
I_{6} + I_{7}
&= \sum^n_{i=2} \int_{V_{i}^{(n)}} b(\overline{v}^{(n)}_i) p_{i-1}^{(n)} \frac{v(\cdot, x) - v(\cdot, x- \Delta x^{(n)})}{\Delta x^{(n)}} dx\\
&\hspace{10mm} - \sum^n_{i=2} \int_{V_{i}^{(n)}} b(\overline{v}^{(n)}_i) p_{i-1}^{(n)} \widetilde{v}^{(n)} dx.
\end{align*}
Thus, we have
\begin{align*}
&\frac{d}{dt} \int^1_0 h'(v^{(n)}) \frac{1}{2} |v - v^{(n)}|^2 dx
+ |\partial_x v - \widetilde{v}^{(n)}|_{H_\Delta^{(n)}}^2
+ \int^{\Delta x^{(n)}}_0 |\partial_x v|^2 dx\\
&= -\int^1_0 \partial_t v (h'(v) - h'(v^{(n)}))) (v - v^{(n)}) dx 
+ \frac{1}{2} \int^1_0 (\partial_t h'(v^{(n)}))|v-v^{(n)}|^2 dx\\
&\hspace{10mm} + \sum^n_{i = 2} \int_{V_{i}^{(n)}} (\partial_x v(\cdot, x_{i-1}^{(n)}) - \partial_x v) \widetilde{v}^{(n)}(\cdot, x) dx
+ \int^1_{\Delta x^{(n)}} \widetilde{v}^{(n)} \left( \frac{v(\cdot, x) - v(\cdot, x-\Delta x^{(n)})}{\Delta x^{(n)}} - \partial_x v \right)dx\\
&\hspace{10mm} -\int^1_{0} b(v) p \partial_x v dx
+ \sum^n_{i=2} \int_{V_{i}^{(n)}} b(v(\cdot, x_{i-1}^{(n)})) p(\cdot, x_{i-1}^{(n)}) \widetilde{v}^{(n)}(\cdot, x) dx\\
&\hspace{10mm} +\sum^n_{i=2} \int_{V_{i}^{(n)}} b(\overline{v}^{(n)}_i) p_{i-1}^{(n)} \frac{v(\cdot, x) - v(\cdot, x- \Delta x^{(n)})}{\Delta x^{(n)}} dx
- \sum^n_{i=2} \int_{V_{i}^{(n)}} b(\overline{v}^{(n)}_i) p_{i-1}^{(n)} \widetilde{v}^{(n)} dx\\
&=: \sum^{15}_{j = 8} I_j
\end{align*}
a.e. on $[0, T]$. H\"{o}lder's inequality gives
\begin{align*}
I_{8}
&\leq C_{h, 2} |\partial_t v|_H \left( \int^{\Delta x^{(n)}}_0 |v - v^{(n)}|^4 dx + \int^1_{\Delta x^{(n)}} |v - v^{(n)}|^4 dx\right)^\frac{1}{2}.
\end{align*}
Thanks to Lemma \ref{lem_GNineq_dc}, we see that
\begin{align*}
&\int^1_{\Delta x^{(n)}} |v - v^{(n)}|^4 dx\\
&\leq C_5 |v - v^{(n)}|_{H_\Delta^{(n)}}^4\\
&\hspace{10mm} + C_5 \Bigg(|\partial_x v - \widetilde{v}^{(n)}|_{H_\Delta^{(n)}}^2 + \int^1_{\Delta x^{(n)}} \left| \int^x_{x-\Delta x^{(n)}} \partial_{\xi}^2 v d\xi \right|^2 dx
+ \Delta x^{(n)} |\partial_x v|_H^2 \Bigg)|v - v^{(n)}|_{H_\Delta^{(n)}}^2\\
&\leq C_5 (|v - v^{(n)}|_{H_\Delta^{(n)}}^4
+|\partial_x v - \widetilde{v}^{(n)}|_{H_\Delta^{(n)}}^2 |v - v^{(n)}|_{H_\Delta^{(n)}}^2
+(\Delta x^{(n)}|\partial_{x}^2 v|_H^2 + \Delta x^{(n)} |\partial_{x} v|_H^2) |v - v^{(n)}|_{H_\Delta^{(n)}}^2).
\end{align*}
With the aid of Lemma \ref{lem_LinfLinfesti}, there exists $C_{7} > 0$, independent of $n$, such that
\begin{align*}
\left( \int^{\Delta x^{(n)}}_0 |v - v^{(n)}|^4 dx \right)^\frac{1}{2}
\leq (|v|_{L^\infty(0, 1)} + |v^{(n)}|_{L^\infty(0, 1)})^4 (\Delta x^{(n)})^\frac{1}{2}
\leq C_{7} (\Delta x^{(n)})^\frac{1}{2}.
\end{align*}
Thus, we obtain
\begin{align*}
I_{8}
&\leq C_{h, 2} |\partial_t v|_H \Bigg[ C_{7} (\Delta x^{(n)})^\frac{1}{2}
+ \sqrt{C_5} \bigg\{ |v - v^{(n)}|_{H_\Delta^{(n)}}^2 
+ |\partial_x v - \widetilde{v}^{(n)}|_{H_\Delta^{(n)}} |v - v^{(n)}|_{H_\Delta^{(n)}}\\
&\hspace{10mm} + (\Delta x^{(n)}|\partial_{x}^2 v|_H^2
 + \Delta x^{(n)} |\partial_x v|_H^2)^\frac{1}{2}
|v - v^{(n)}|_{H_\Delta^{(n)}} \bigg\} \Bigg]\\
&\leq \frac{1}{8} |\partial_x v - \widetilde{v}^{(n)}|_{H_\Delta^{(n)}}^2
+(\Delta x^{(n)})^\frac{1}{2} C_{h, 2} C_{7} |\partial_t v|_H
+ \Delta x^{(n)} |\partial_{x}^2 v|_H^2 \\
&\hspace{10mm} + \Delta x^{(n)} |\partial_x v|_H^2
+ \left( \frac{5C_{h, 2}^2 C_5}{2}|\partial_t v|_H^2 + 1\right) |v - v^{(n)}|_{H_\Delta^{(n)}}^2.
\end{align*}
In the same way, we have
\begin{align*}
I_{9}
&\leq \frac{C_{h, 2}}{2} \int^1_0 |\partial_t v^{(n)}| |v - v^{(n)}|^2 dx\\
&\leq \frac{1}{8} |\partial_x v - \widetilde{v}^{(n)}|_{H_\Delta^{(n)}}^2
+(\Delta x^{(n)})^\frac{1}{2} \frac{C_{h, 2} C_{7}}{2} |\partial_t v^{(n)}|_H
+ \Delta x^{(n)} |\partial_{x}^2 v|_H^2 \\
&\hspace{10mm} + \Delta x^{(n)} |\partial_x v|_H^2
+ \left( \frac{5C_{h, 2}^2 C_5}{8} |\partial_t v^{(n)}|_H^2 + 1\right) |v - v^{(n)}|_{H_\Delta^{(n)}}^2.
\end{align*}
It follows from Lemma \ref{lem_dtvL2L2esti} that
\begin{align*}
I_{10}
&\leq \sum^n_{i = 2} \int_{V_{i}^{(n)}} |\partial_x v(\cdot, x_{i-1}^{(n)}) - \partial_x v| |\widetilde{v}^{(n)}| dx\\
&= \sum^n_{i = 2} \int_{V_{i}^{(n)}} \left| \int^x_{x_{i-1}} \partial_{y}^2 v dy \right| |\widetilde{v}^{(n)}| dx\\
&\leq \sum^n_{i = 2} \int_{V_{i}^{(n)}} |\partial_{y}^2 v| dy \int_{V_{i}^{(n)}} |\widetilde{v}^{(n)}| dx\\
&\leq (\Delta x^{(n)})^\frac{1}{2}|\partial_{x}^2 v|_H \int^1_{\Delta x^{(n)}} |\widetilde{v}^{(n)}| dx\\
&\leq (\Delta x^{(n)})^\frac{1}{2} 4C_2 |\partial_{x}^2 v|_H.
\end{align*}
The mean value theorem yields
\begin{align}
I_{11}
&\leq \int^1_{\Delta x^{(n)}} |\widetilde{v}^{(n)}| \left| \frac{v(\cdot, x) - v(\cdot, x-\Delta x^{(n)})}{\Delta x^{(n)}} - \partial_x v \right| dx \nonumber\\
&\leq \int^1_{\Delta x^{(n)}} |\widetilde{v}^{(n)}| \int^1_0 |\partial_{\xi}^2 v| d\xi dx \nonumber\\
&\leq (\Delta x^{(n)})^\frac{1}{2} \int^1_{\Delta x^{(n)}} |\widetilde{v}^{(n)}| \left( \int_0^1 |\partial_{\xi}^2 v|^2 d\xi \right)^\frac{1}{2} dx \nonumber\\
&\label{esti6} \leq (\Delta x^{(n)})^\frac{1}{2} 4C_2 |\partial_{x}^2 v|_H.
\end{align}
We next transform $I_{12} - I_{15}$ as follows:
\begin{align*}
I_{13}
&= \sum^n_{i=2} \int_{V_{i}^{(n)}} (b(v(\cdot, x_{i-1}^{(n)})) - b(v)) p(\cdot, x_{i-1}^{(n)}) \widetilde{v}^{(n)} dx\\
&\quad +\sum^n_{i=2} \int_{V_{i}^{(n)}} b(v) p(\cdot, x_{i-1}^{(n)}) \widetilde{v}^{(n)} dx\\
&=: I_{16} + I_{17},\\
I_{14}
&= \sum^n_{i=2} \int_{V_{i}^{(n)}} b(\overline{v}^{(n)}_i) p_{i-1}^{(n)} \left(\frac{v(\cdot, x) - v(\cdot, x- \Delta x^{(n)})}{\Delta x^{(n)}} - \partial_x v \right) dx\\
&\quad + \sum^n_{i=2} \int_{V_{i}^{(n)}} b(\overline{v}^{(n)}_i) p_{i-1}^{(n)} \partial_x v dx\\
&=: I_{18} + I_{19},\\
I_{12} + I_{17}
&= \int^{\Delta x^{(n)}}_0 b(v)p \partial_x v dx
+\sum^n_{i=2} \int_{V_{i}^{(n)}} b(v) (p(\cdot, x_{i-1}^{(n)}) - p) \partial_x v dx\\
&\quad + \sum^n_{i=2} \int_{V_{i}^{(n)}} b(v) p(\cdot, x_{i-1}^{(n)}) (\widetilde{v}^{(n)} - \partial_x v) dx\\
&=: I_{20} + I_{21} + I_{22},\\
I_{15} + I_{19} + I_{22}
&= \sum^n_{i=2} \int_{V_{i}^{(n)}} (b(v) p(\cdot, x_{i-1}^{(n)}) - b(\overline{v}^{(n)}_i) p_{i-1}^{(n)}) (\widetilde{v}^{(n)} - \partial_x v) dx\\
&= \sum^n_{i=2} \int_{V_{i}^{(n)}} (b(v) - b(\overline{v}^{(n)}_i)) p(\cdot, x_{i-1}^{(n)}) (\widetilde{v}^{(n)} - \partial_x v) dx\\
&\quad + \sum^n_{i=2} \int_{V_{i}^{(n)}} b(\overline{v}^{(n)}_i) (p(\cdot, x_{i-1}^{(n)}) - p_{i-1}^{(n)}) (\widetilde{v}^{(n)} - \partial_x v) dx\\
&\eqqcolon I_{23} + I_{24}
\end{align*}
a.e. on $[0, T]$.
We set $C_{8} \coloneqq C_{b, 2} |p|_{L^\infty(Q(T))}$.
From H\"older's inequality, we see that
\begin{align*}
I_{16}
&\leq C_{8} \sum^n_{i=2} \int_{V_{i}^{(n)}} |v - v(\cdot, x_{i-1}^{(n)})| |\widetilde{v}^{(n)}| dx\\
&\leq C_{8} \sum^n_{i=2} \int_{V_{i}^{(n)}} \left| \int^x_{x_{i-1}} \partial_y v dy \right| |\widetilde{v}^{(n)}| dx\\
&\leq (\Delta x^{(n)})^\frac{1}{2} C_{8} |\partial_x v|_H \int^1_{\Delta x^{(n)}} |\widetilde{v}^{(n)}| dx\\
&\leq (\Delta x^{(n)})^\frac{1}{2} 4 C_2 C_{8} |\partial_x v|_H.
\end{align*}
By the same argument as in \eqref{esti6}, we get
\begin{align*}
I_{18}
&\leq C_{8} \int^1_{\Delta x^{(n)}} \left|\frac{v(\cdot, x) - v(\cdot, x- \Delta x^{(n)})}{\Delta x^{(n)}} - \partial_x v \right| dx\\
&\leq C_{8} \int^1_{\Delta x^{(n)}}  \int_0^1 |\partial_{\xi}^2 v| d\xi dx\\
&\leq (\Delta x^{(n)})^\frac{1}{2} C_{8} |\partial_{x}^2 v|_H.
\end{align*}
Applying Lemma \ref{lem_1DGNineq}, we have
\begin{align*}
I_{20} \leq \Delta x^{(n)} \frac{C_8}{2} (3 |\partial_x v|_H^2 + |\partial_{x}^2 v|_H^2).
\end{align*}
The definition of $p_i^{(n)}$ implies
\begin{align*}
I_{21}
&\leq C_{b, 2} \sum^n_{i=2} |p(\cdot, x_{i-1}^{(n)}) - p_{i-1}^{(n)}| \int_{V_{i}^{(n)}} |\partial_x v| dx\\
&\leq C_{b, 2} \frac{1}{\Delta x^{(n)}} \sum^n_{i=2} \left|\int_{V_{i-1}^{(n)}} p(\cdot, x_{i-1}^{(n)}) - p(\cdot, \xi) d\xi \right| \int_{V_{i}^{(n)}} |\partial_x v| dx\\
&\leq C_{b, 2} \frac{1}{\Delta x^{(n)}} \sum^n_{i=2} \left|\int_{V_{i-1}^{(n)}} \int^{x_{i-1}}_\xi \partial_y p dy d\xi \right| \int_{V_{i}^{(n)}} |\partial_x v| dx\\
&\leq C_{b, 2} (\Delta x^{(n)})^\frac{1}{2} \sum^n_{i=2} \int_{V_{i-1}^{(n)}} |\partial_x p| dx \left\{ \int_{V_{i}^{(n)}} |\partial_x v|^2 dx \right\}^\frac{1}{2}\\
&\leq C_{b, 2} \Delta x^{(n)} \sum^n_{i=2} \left\{ \int_{V_{i-1}^{(n)}} |\partial_x p|^2 dx \right\}^\frac{1}{2} \left\{ \int_{V_{i}^{(n)}} |\partial_x v|^2 dx \right\}^\frac{1}{2}\\
&\leq \Delta x^{(n)} \frac{C_{b, 2}^2}{2} (|\partial_x p|_H^2 + |\partial_x v|_H^2).
\end{align*}
Similarly, we obtain
\begin{align*}
I_{24} 
&\leq C_{b, 2} \sum^n_{i=2} |p(\cdot, x_{i-1}^{(n)}) - p_{i-1}^{(n)}| \int_{V_{i}^{(n)}} |\partial_x v - \widetilde{v}^{(n)}| dx\\
&\leq C_{b, 2} \Delta x^{(n)} \sum^n_{i=2} \left\{ \int_{V_{i}^{(n)}} |\partial_x p|^2 dx \right\}^\frac{1}{2} \left\{ \int_{V_{i}^{(n)}} |\partial_x v - \widetilde{v}^{(n)}|^2 dx \right\}^\frac{1}{2}\\
&\leq \frac{1}{8} |\partial_x v - \widetilde{v}^{(n)}|_{H_\Delta^{(n)}}^2 + 2C_{b, 2}^2 (\Delta x^{(n)})^2|\partial_x p|_H^2.
\end{align*}
From Young's inequality, we see that
\begin{align*}
I_{23}
&\leq C_{8} \int^1_{\Delta x^{(n)}} \left|\frac{v^{(n)}(\cdot - \Delta x^{(n)}) + v^{(n)}}{2} - v \right| |\partial_x v - \widetilde{v}^{(n)}| dx\\
&\leq \frac{1}{8} |\partial_x v - \widetilde{v}^{(n)}|_{H_\Delta^{(n)}}^2 + \frac{C_{8}^2}{2} |v^{(n)}(\cdot - \Delta x^{(n)}) + v^{(n)}- 2v |_{H_\Delta^{(n)}}^2\\
&\leq \frac{1}{8} |\partial_x v - \widetilde{v}^{(n)}|_{H_\Delta^{(n)}}^2 + \frac{9C_{8}^2}{2} \Bigg( |v - v^{(n)}|_{H_\Delta^{(n)}}^2 
+ \int^1_{\Delta x^{(n)}} |v - v(\cdot - \Delta x^{(n)})|^2 dx\\
&\hspace{10mm} + \int^1_{\Delta x^{(n)}} |v^{(n)}(\cdot - \Delta x^{(n)}) - v(\cdot - \Delta x^{(n)})|^2 dx \Bigg) \\
&\leq \frac{1}{8} |\partial_x v - \widetilde{v}^{(n)}|_{H_\Delta^{(n)}}^2 + 9C_{8}^2 |v - v^{(n)}|_H^2
+ \Delta x^{(n)} \frac{9C_{8}^2}{2} |\partial_x v|_H^2.
\end{align*}
From these uniform estimates, it follows that there exists a positive constant $C_{9}$ satisfying 
\begin{align*}
&\frac{d}{dt} \int^1_0 h'(v^{(n)}) |v - v^{(n)}|^2 dx
+ |\partial_x v - \widetilde{v}^{(n)}|_{H_\Delta^{(n)}}^2\\
&\leq (\Delta x^{(n)})^\frac{1}{2} C_{9} (|\partial_t v|_H^2 + |\partial_t v^{(n)}|_H^2 + |\partial_{x}^2 v|_H^2 + |\partial_x p|_H^2 + 1)
+ C_9 |v - v^{(n)}|_{H_\Delta^{(n)}}^2 (|\partial_t v|_H^2 + |\partial_t v^{(n)}|_H^2 + 1)
\end{align*}
a.e. on $[0, T]$.
Gronwall's inequality gives
\begin{align*}
&\int^1_0 h'(v^{(n)}(t)) |v(t) - v^{(n)}(t)|^2 dx
+ \int^t_0 |\partial_x v - \widetilde{v}^{(n)}|_{H_\Delta^{(n)}}^2 d\tau\\
&\leq C_{10} \left\{ \int^1_0 h'(v^{(n)}(0, x)) |(v - v^{(n)})(0, x)|^2 dx    \right. \\ 
&\hspace{10mm} \left. + (\Delta x^{(n)})^\frac{1}{2} C_{9} \left( \int^T_0 |\partial_t v|_H^2 dt + \int^T_0 |\partial_t v^{(n)}|_H^2 dt + \int^T_0 |\partial_{x}^2 v|_H^2 dt + \int^T_0 |\partial_x p|_H^2 dt + T\right) \right\}
\end{align*}
for any $t \in [0, T]$, where 
\begin{align*}
C_{10} \coloneqq \exp \left\{ \frac{C_{9}}{C_{h, 1}} \left( \int^T_0 |\partial_t v|_H^2 dt + \frac{2C_2}{C_{h, 1}} + T \right) \right\}.
\end{align*} 
With the aid of the definition of $v_{0, i}^{(n)}$, we get
\begin{align*}
\int^1_0 h'(v^{(n)}(0, x)) |(v - v^{(n)})(0, x)|^2 dx
&\leq C_{h, 2} \sum^n_{i = 1} \int_{V_{i}^{(n)}} |v_0(x) - v_{0, i}^{(n)}|^2 dx\\
&= C_{h, 2} \sum^n_{i = 1} \int_{V_{i}^{(n)}} \left| v_0(x) - \frac{1}{\Delta x^{(n)}} \int_{V_{i}^{(n)}} v_0(\xi) d\xi \right|^2 dx\\
&= \frac{C_{h, 2}}{(\Delta x^{(n)})^2} \sum^n_{i = 1} \int_{V_{i}^{(n)}} \left| \int_{V_{i}^{(n)}} (v_0(x) - v_0(\xi)) d\xi \right|^2 dx\\
&\leq \frac{C_{h, 2}}{(\Delta x^{(n)})^2} \sum^n_{i = 1} \int_{V_{i}^{(n)}} \left( \int_{V_{i}^{(n)}} \left| \int^x_\xi \partial_y v_{0} dy \right| d\xi \right)^2 dx\\
&\leq \frac{C_{h, 2}}{\Delta x^{(n)}} \sum^n_{i = 1} \int_{V_{i}^{(n)}} \left( \int_{V_{i}^{(n)}} \left| \int^x_\xi |\partial_y v_{0}|^2 dy \right|^\frac{1}{2} d\xi \right)^2 dx\\
&\leq \Delta x^{(n)} C_{h, 2} |\partial_x v_{0}|_H^2.
\end{align*}
Consequently, we arrive at the desired conclusion.
\end{proof}

\section*{Acknowledgement}
 
This work is partially supported by Ebara Corporation.
The authors are grateful to members of Ebara Corporation for showing several interesting phenomena related to moisture penetration with fruitful discussion, which triggered this work.


\begin{thebibliography}{999}
\bibitem{B-D}
M. Berardi, F. V. Difonzo, {\it Strong solutions for Richards' equation with Cauchy conditions and constant pressure gradient}, Environ. Fluid Mech. {\bf 20}(2020), 165--174.

%\bibitem{C-T-T-Y}
%Y. Chiyo, H. Terasaki, Y. Tsuzuki, T. Yokota, {\it Global existence and uniqueness of weak solutions to a one-dimensional moisture transport model for porous materials}, Evol. Equ. Control Theory, {\bf 14}(2025), 1614--1637.

\bibitem{D-S-V}
V. Dolej\v{s}\'{i}, H.-G. Shin, M. Vlas\'ak, {\it Error estimates and adaptivity of the space-time discontinuous Galerkin method for solving the Richards equation}, J. Sci. Comput. {\bf 101}(2024), no.~1, Paper No. 11. %39 pp.

\bibitem{E-G-H}
R. Eymard, T. Gallou\"{e}t, R. Herbin, {\it Finite volume methods}, Handbook of Numerical Analysis, {\bf 7}(2000).

\bibitem{F-I-H-O}
K. Fukui, C. Iba, S. Hokoi, D. Ogura, {\it Effect of air pressure on moisture transfer inside porous building materials}, J. Environ. Eng., AIJ, {\bf 83}(2018), 39--47. 

\bibitem{G-D-W-P}
D. W. Green, H. Dabiri, C. F. Weinaug, R. Prill, {\it Numerical modeling of unsaturated groundwater flow and comparison of the model to a field experiment}, Water Resources Research, {\bf 6}(1970), 862--874.

\bibitem{M-R}
W. Merz, P. Rybka, {\it Strong solutions to the Richards equation in the unsaturated zone}, J. Math. Anal. Appl., {\bf 371}(2010), 741--749.

\bibitem{M-V}
K. Mitra, M. Vohral\'ik, {\it A posteriori error estimates for the Richards equation}, Math. Comp., {\bf 93}(2024), 1053--1096.

\bibitem{M-A-1}
A. Morimura, T. Aiki, {\it Convergence of approximate solutions constructed by the finite volume method for the moisture transport model in porous media}, Adv. Math. Sci. Appl., {\bf 35}(2026), 163--185.
%arXiv: 2505.09763.

\bibitem{M-A-2}
A. Morimura, T. Aiki, {\it Solvability of the moisture transport model for porous materials}, Adv. Math. Sci. Appl., {\bf 33}(2024), 501--524.

\bibitem{Roubicek}
T. Roub\'{i}\v{c}ek, {\it Nonlinear partial differential equations with applications}, Internat. Ser. Numer. Math., {\bf 153}, Birkh\"{a}user Verlag, Basel (2005).

\bibitem{S-M-S-B-R}
J. S. Stokke, K. Mitra, E. Storvik, J. W. Both, F. A. Radu, {\it An adaptive solution strategy for Richards' equation}, Comput. Math. Appl., {\bf 152}(2023), 155--167.

\end{thebibliography}
\end{document}